\documentclass[12pt,a4paper]{article}
\usepackage{amsthm,amsfonts,amsmath,amssymb,bookmark}
\usepackage{color}

\theoremstyle{plain}
\newtheorem{theorem}{Theorem}
\newtheorem{lemma}[theorem]{Lemma}

\theoremstyle{definition}

\theoremstyle{remark}

\def\bq{\begin{eqnarray}}
\def\eq{\end{eqnarray}}
\def\bqq{\begin{eqnarray*}}
\def\eqq{\end{eqnarray*}}
\def\nn{\nonumber}

\def\eps{\varepsilon}

\def\R{\mathbb{R}}

\def\cE {\mathcal{E}}
\def \d {{\rm d}}

\title{Ground state solution of a Kirchhoff type equation with singular potentials}

\author{Thanh Viet Phan \footnote{Institute of Fundamental and Applied Sciences, Duy Tan University, Ho Chi Minh City 700000, Vietnam} \footnote{Faculty of Natural Sciences, Duy Tan University, Danang City 550000, Vietnam.} \footnote{Email: phanthanhviet4@duytan.edu.vn}
}
\date{\normalsize\today}
\begin{document}

\maketitle


\begin{abstract}  
We study the existence and blow-up behavior of minimizers for 
\bqq
E(b)=\inf\Big\{\cE_b(u) \,|\, u\in H^1(\R^2), \|u\|_{L^2}=1\Big\},
\eqq
here $\cE_b(u)$ is the Kirchhoff energy functional defined by
$$
\cE_b(u)= \int_{\R^2} |\nabla u|^2 \d x + b \left( \int_{\R^2} |\nabla u|^2\d x \right)^2 + \int_{\R^2} V(x) |u(x)|^2 \d x - \frac{a}{2} \int_{\R^2} |u|^4\d x,
$$
where $a>0$ and $b>0$ are constants. When $V(x)= -|x|^{-p}$ with $0<p<2$, we prove that the problem has (at least) a minimizer that is non-negative and 
radially symmetric decreasing.  For $a\ge a^*$ (where $a^*$ is the optimal constant in the Gagliardo-Nirenberg inequality), we get the behavior of $E(b)$ when $b\to$ $0^+$. Moreover, for the case $a=a^*$, we analyze the details of the behavior of the minimizers $u_b$ when $b\to 0^+$.
\bigskip

   
\end{abstract}


\section{Introduction}

Let $a>0$ and $b>0$ are constants . Consider the minimization problem 
\bq \label{eq:GP}
E(b)=\inf\Big\{\cE_b(u) \,|\, u\in H^1(\R^2), \|u\|_{L^2}=1\Big\},
\eq
where $\cE_b(u)$ is the Kirchhoff energy functional defined by
\bq\label{eq:KC}
\cE_b(u)= \int_{\R^2} |\nabla u|^2\d x  + b \left( \int_{\R^2} |\nabla u|^2\d x \right)^2 + \int_{\R^2} V(x) |u(x)|^2 \d x - \frac{a}{2} \int_{\R^2} |u|^4\d x.
\eq
Since $|\nabla u|\ge |\nabla|u||$ pointwise, in consideration of \eqref{eq:GP} we can always restrict to $u\ge 0$. 

\subsection{The Gross-Pitaevskii theory} If $b=0$, then \eqref{eq:KC} becomes the 2D Gross-Pitaevskii energy functional
\bq\label{eq:GP1}
\cE_0(u)= \int_{\R^2} |\nabla u|^2\d x  + \int_{\R^2} V(x) |u(x)|^2 \d x - \frac{a}{2} \int_{\R^2} |u|^4\d x.
\eq
We define the GP energy to be 
\bqq 
E(0)=\inf\Big\{\cE_0(u) \,|\, u\in H^1(\R^2), \|u\|_{L^2}=1\Big\}.
\eqq

The emergence of the Gross-Pitaevskii  energy functional from Schr\"odinger quantum mechanics is well-known, see e.g. \cite{LewNamRou-15}. Moreover, the 
Gross-Pitaevskii energy functional is important in the study of Bose-Einstein condensates (BEC), see e.g \cite{BraSacTolHul-95,KagMurShl-98,SacStoHul-98}.  For this reason, problem \eqref{eq:GP} with $b=0$ has received a lot of interest in mathematics in recent years. 

When $V=0$, by defining $u_\ell(x)=\ell u(\ell x)$, we have the scaling property
\bqq
\cE_0(u_\ell)=\ell^2 \cE_0(u),~~ \forall \ell>0. 
\eqq
Thus $E(0)=-\infty$ if $a>a^*$ and $E(0)=0$ if $a\le a^*$. Here $a^*>0$ is the optimal constant in the Gagliardo-Nirenberg inequality:
\bq 
\label{eq:GN} 
\int_{\R^2} |\nabla u(x)|^2 \d x \ge \frac{a^*}{2} \int_{\R^2} |u(x)|^4  \d x, \quad \forall u\in H^1(\R^2), \|u\|_{L^2}=1.
\eq
It is well-known that (see e.g. \cite{GidNiNir-81,Weinstein-83,MclSer-87})
\bq \label{eq:a*-Q}
a^*=\int_{\R^2} |Q|^2 \d x =\int_{\R^2} |\nabla Q|^2\d x=\frac{1}{2}\int_{\R^2} |Q|^4\d x,
\eq
where $Q\in H^1(\R^2)$ is the positive solution to the nonlinear Schr\"odinger equation \bq \label{eq:Q}
-\Delta Q + Q - Q^3 =0, \quad Q\in H^1(\R^2).
\eq 
Moreover, $Q$ is unique (up to translations and dilations) and it can be chosen to be radially symmetric decreasing. Therefore, when $V=0$, $E(0)$ has minimizers if and only if $a=a^*$, and all minimizers are of the form $\beta Q_0 (\beta x-x_0)$ with $Q_0=Q/\|Q\|_{L^2}$, $\beta >0$ and $x_0\in \R^2$.

When $V\ne 0$, the situation changes crucially.  By changing the property of the potential $V(x)$, the existence and blow-up behavior of minimizers for $E(0)$ have been studied in many
works. In the case of trapping potential, i.e. 
\bqq
V\in L^{\infty}_{loc}(\R^2), \lim_{|x|\to \infty}V(x)=\infty\text{ and }\inf_{x\in \R^2}V(x)= 0.
\eqq
In \cite{GuoSei-14}, Guo and Seiringer showed that $E(0)$ has a minimizer if and only if $a<a^*$. Moreover, they prove that if
$$V (x) = h(x) \prod_{j=1}^J |x-x_j|^{p_j}, \quad 0<C^{-1}\le h(x) \le C, $$
then when $a\uparrow a^*$, up to subsequences of $\{u_a\}$, there exists $i_0\in \{1,2,...,J\}$ such that 
$$p_{i_0}=\max \{p_j: 1\le j\le J \} , \quad h(x_{i_0})=\min\{h(x_j):p_j=p_{i_0}\}$$
and  
\bqq 
\lim_{a\uparrow a^*} \eps_a u_a(x_{i_0}+\eps_a x) = \beta Q_0(\beta x)
\eqq
strongly in $L^q(\R^2)$ for all $q\in [2,\infty)$, where 
\bqq
\eps_a=(a^*-a)^{1/(p_{i_0}+2)}~~ and ~~  
\beta =  \left( \frac{p_{i_0} h(x_{i_0})}{2} \int_{\R^2} |x|^p |Q(x)|^2 \d x \right)^{1/(p_{i_0}+2)} .
\eqq

Thus if $V(x)$ behaves as a homogeneous potential, at least locally, it is possible to extract the details of the blow-up profile when $a$ tends to $a^*$. This result has been extended to other kinds of trapping potentials \cite{DenGuoLu-15,GuoWangZengZhou-15,GuoZenZho-16}, periodic potentials \cite{WanZha-17}, general bounded potentials \cite{Phan-19} and singular potentials \cite{Phan-17}.
 
In particular, in \cite{Phan-17}, we prove that if $V$ is the singular potential, i.e.
\bqq
V(x)=-\frac{1}{|x|^p},~~ 0<p<2,
\eqq
then $E(0)$ has a minimizer if $a\in (a_*,a^*)$ for some constant $a_*<a^*$. Moreover, a blow-up result similar to \cite{GuoSei-14} holds true, but the concentration of the minimizer at the singular point of $V$ results a huge  cancellation between the kinetic and potential energies, making the analysis more challenging.

\subsection{The Kirchhoff theory} 
Now we turn to the case when $b>0$, which is the main focus of the present paper. When $b>0$, the minimization problem \eqref{eq:GP}  arises in studying the Kirchhoff equation
\bqq
-\left(1+b\int_{\R^2}|\nabla u|^2 dx \right) \Delta u+V(x)u=a|u|^2 u +\mu u, ~~ x\in \R^2, \mu\in \R.
\eqq
The appearance of the nonlocal term $\left( \int_{\R^2} |\nabla u|^2\d x \right)^2$ in the energy functional \eqref{eq:KC}, which causes some mathematical difficulties, has attracted so much attention, see e.g. \cite{LiYe-19,GuoZho-21,HuTan-21} and the references therein.

In \cite{GuoZhaZho-18}, Guo, Zhang and Zhou proved that if $V (x)$ is the trapping potential satisfying
\bq \label{eq:VIn}
0\ne V\in C(\R^2), \lim_{|x|\to \infty}V(x)=\infty\text{ and }\inf_{x\in \R^2}V(x)= 0,
\eq
then $E(b)$ has a minimizer for all $a> 0$ and $b>0$. Thus when $a\ge a^*$, the existence of minimizers changes completely when $b>0$ (in comparison to the result in \cite{GuoSei-14} that $E(0)$ has no minimizer for all $a\ge a^*$). The next natural question is what would happen for the minimizer $u_b$ of $E(b)$ if $b\to 0^+$ and $a\ge a^*$ ? 

In a very recent paper, with the condition $V$ is satisfied \eqref{eq:VIn}, Guo and Zhou \cite{GuoZho-21} analyzed the blow-up behavior of the minimizer for problem \eqref{eq:GP} as $b\to 0^+$ when $a\ge a^*$. Furthermore, under some homogeneous conditions about $V(x)$, they prove that, for $a\ge a^*$, there exists a unique nonnegative minimizer for $E(b)$ as $b > 0$ being small enough.

\subsection{Main results}

In the present paper, we are interested in the minimization problem \eqref{eq:GP}, with $b>0$, in the case of singular potentials which are unbounded from below, e.g.

\bq\label{eq:V}
V(x)=-\frac{1}{|x|^p},~~ \text{for some constant }0<p<2 .
\eq

In the special case $p=1$, this singular potential corresponds to the gravitational attraction or the Coulombic attraction. The case of singular potential is interesting because of the instability of $E(b)$ when $b \to 0^+$  which is not observed in the case of trapping potentials.

First, we have a result on the existence of minimizers.

\begin{theorem}[Existence] \label{thm:1} Let $V$ as in 
\eqref{eq:V}. Then for any $b>0$ and $a>0$, the problem \eqref{eq:GP} has (at least) a minimizer that is non-negative and radially symmetric decreasing.
\end{theorem}

In the case $b=0$, the result in \cite[Theorem 1]{Phan-17} concludes that $E(b)$ has a minimizer if $a<a^*$ and $a$ is close to $a^*$ sufficiently. The reason is that most of the kinetic energy is canceled by the interaction energy, and hence the condensate is trapped by any small negative well. However, the technique of the proof in \cite[Theorem 1]{Phan-17} cannot be applied to the case $b>0$ from the appearance of the nonlocal term $\left( \int_{\R^2} |\nabla u|^2\d x \right)^2$.  To overcome this difficulty, we first prove $\cE_b(u) \ge \cE_b(u^*)$, where $u^*$ is the symmetric-decreasing rearrangement of $u$. In this place, we use the assumption that $V$ is radially symmetric decreasing to simplify the analysis. Therefore, from the fact that $E(b)>-\infty$, we can find a minimizing sequence $\{u_n\}$ such that $u_n=u_n^*$ , i.e $u_n$ is radially symmetric decreasing. Finally, we use the compact embedding result due to Strauss \cite{STRAUSS} (see also \cite{BerLio-83}) to get the conclusion.

Comparing this result with the result in \cite[Lemma 3]{Phan-17} that $E(b)$ has no minimizer if $b=0$ and $a\ge a^*$, a natural question is what would happen for $E(b)$ and the minimizers $u_b$ of $E(b)$ when $b\to 0^+$ and $a\ge a^*$ ? 

In the critical case $a=a^*$, the behavior of the energy $E(b)$ and the corresponding minimizer $u_b$ when $b\to 0^+$ is given below.

\begin{theorem}[Blow-up] \label{thm:2} Let $V$ as in \eqref{eq:V}. Suppose that $a=a^*$. Let $u_b$ be the minimizer for $E(b)$. We have

\bqq
\lim_{b\to 0} b^{p/(4-p)}E(b)=\left(\int_{\R^2}\frac{|Q_0(x)|^2}{|x|^p} \d x\right)^{4/(4-p)}\left[\left(\frac{p}{4}\right)^{4/(4-p)}-\left(\frac{p}{4}\right)^{p/(4-p)}\right].
\eqq
Moreover, for every sequence $b_n\to 0^+$, there exists a subsequence (still denoted by $b_n$) such that
\bqq
\lim_{n\to \infty}\eps_n u_{b_n}\left(\eps_n x\right)=\beta Q_0(\beta x)
\eqq
strongly in $H^1(\R^2)$, where 
\bqq
\eps_n = b_n^{1/(4-p)} ~~ and ~~ \beta= \left(\frac{p}{4} \int_{\R^2} \frac{|Q_0(x)|^2}{|x|^p} \d x\right)^{1/(4-p)}.
\eqq
\end{theorem}

Finally, in the case $a>a^*$, although we are unable to determine the blow-up profile of the minimizer $u_b$, we can obtain the behavior of $E(b)$ as $b\to 0^+$ as follows.

\begin{theorem} \label{thm:3} Let $V$ as in \eqref{eq:V}. Suppose that $a>a^*$. We have
\bq
\lim_{b\to 0}  bE(b)=-\frac{1}{4}\left(\frac{a}{a^*}-1\right)^2.
\eq
\end{theorem}

Our proof of the blow-up results is based on a careful implementation of the energy method, in which the upper bound and lower bound of E(b) are done separately. In both directions (upper and lower bounds), the cancellation to the leading order of the kinetic and potential energies makes our analysis different from that of \cite{GuoZho-21}. The idea of treating singular potentials from \cite{Phan-17} will be helpful for us. However, due to the presence of the nonlocal term $\left( \int_{\R^2} |\nabla u|^2\d x \right)^2$ in the energy functional, we have to encounter more complicated and technical calculations than in \cite{Phan-17}. In particular, we have to deal with the limit $b\to 0^+$ instead of $a\to a^*$ which makes the analysis substantially different.

We will prove Theorem \ref{thm:1} in Section \ref{sec:thm1} and prove Theorems \ref{thm:2} and \ref{thm:3} in Section \ref{Blow-up behavior}.

\section{Existence of minimizers} \label{sec:thm1}
In this section we prove Theorem \ref{thm:1}. As a preliminary step, we have
\begin{lemma} \label{lem:4}Let $V(x)= -|x|^{-p}$ with $0<p<2$. For all $u\in H^1(\R^2)$ and $\|u\|_{L^2}=1$, we have
\bq\label{eq:7lem}
\eps\int_{\R^2} |\nabla u|^2+\int_{\R^2} V |u|^2 \ge -C_p\eps^{-p/(2-p)},~~\forall \eps>0,
\eq
where $C_p$ is a constant depending only on $p$.
\end{lemma}

\begin{proof}
We first show that
\bq\label{eq:cp}
\int_{\R^2} |\nabla u|^2+\int_{\R^2} V |u|^2 \ge -C_p,~~ \forall u\in H^1(\R^2) ,\|u\|_{L^2}=1.
\eq
Let $f=|x|^{-p}\chi_{\{|x|\le 1\}},$ $g=|x|^{-p}\chi_{\{|x|\ge 1\}}$, where $\chi_A$ is the characteristic function of the set $A\subset\R^2$. We have $f\in L^{1/p+1/2}$, $g\in L^{4/p}$ and $V=-(f+g)$.
Let $q=1/p+1/2$, we have $1<q<+\infty$ and $f\in L^q$. Let $q'\in (1,+\infty)$ such that $\frac{1}{q}+\frac{1}{q'}=1$. By  H\"older's and  Sobolev's inequalities , we have 
\begin{align*}
\int_{\R^2} f|u|^2 &=\int_{\R^2} f|u|^2\chi_{\{f\ge \lambda\}}+\int_{\R^2} f|u|^2\chi_{\{f\le \lambda\}}\\
&\le \left(\int_{\R^2}\left|f\chi_{\{f\ge \lambda\}}\right|^q\right)^{1/q}\left(\int_{\R^2}|u|^{2q'}\right)^{1/q'}+\lambda\\
&\le \left(\int_{\R^2}\left|f\chi_{\{f\ge \lambda\}}\right|^q\right)^{1/q}\|u\|_{H^1}^2+\lambda,
\end{align*}
where $\lambda>0$ is a constant.

By the dominated convergence theorem, we have $\left(\int_{\R^2}\left|f\chi_{\{f\ge \lambda\}}\right|^q\right)^{1/q}\to 0 
$ when $\lambda \to +\infty$. Thus from the above estimates, there exists a constant $C_p$ depending only on $p$ such that
\begin{align*}
\int_{\R^2} f|u|^2&\le \frac{1}{2}\|u\|_{H^1}^2+C_p\\
&=\frac{1}{2}\left(\int_{\R^2} |\nabla u|^2+1\right)+C_p.
\end{align*}
Similarly, we obtain that 
\bqq
\int_{\R^2} g|u|^2\le \frac{1}{2}\left(\int_{\R^2} |\nabla u|^2+1\right)+C_p.
\eqq
Therefore,
\bqq \int_{\R^2}|\nabla u|^2+\int_{\R^2} V |u|^2 = \int_{\R^2} |\nabla u|^2 - \int_{\R^2} f |u|^2 -\int_{\R^2} g |u|^2\ge -C_p.
\eqq
Now we prove the inequality \eqref{eq:7lem}. Let $u_\ell(x)=\ell^{-1}u(\ell^{-1}x)$, we have $u_\ell\in H^1(\R^2)$, $\|u_\ell\|_{L^2}=1$ and $u(x)=\ell u _\ell(\ell x)$. Thus

\bq\label{eq:8lem}
\eps\int_{\R^2} |\nabla u|^2+\int_{\R^2} V |u|^2 =\eps\ell^2\int_{\R^2} |\nabla u_\ell|^2+\ell^p\int_{\R^2} V |u_\ell|^2.
\eq
Using \eqref{eq:8lem} with $\ell=\eps^{-1/(2-p)}$ and the inequality \eqref{eq:cp}, we obtain
\bqq
\eps\int_{\R^2} |\nabla u|^2+\int_{\R^2} V |u|^2 =\eps^{-p/(2-p)}\left[\int_{\R^2} |\nabla u_\ell|^2+\int_{\R^2} V |u_\ell|^2\right]\ge -C_p\eps^{-p/(2-p)}.
\eqq
\end{proof}

\begin{proof}[Proof of Theorem \ref{thm:1}] From Lemma \ref{lem:4} and \eqref{eq:GN}, for all $u\in H^1(\R^2)$ and $\|u\|_{L^2}=1$, we have 
\bq\label{eq1:thm1}
 \cE_b(u) \ge b \left( \int_{\R^2} |\nabla u|^2 \right)^2 - \frac{a}{a^*} \int_{\R^2} |\nabla u|^2 -C_p.
\eq
Moreover, by AM-GM inequality, we have
\bq\label{eq2:thm2}
\frac{b}{2} \left( \int_{\R^2} |\nabla u|^2 \right)^2-\frac{a}{a^*} \int_{\R^2} |\nabla u|^2\ge -\frac{1}{2b}\left(\frac{a}{a^*}\right)^2.
\eq
Thus, \eqref{eq1:thm1} and \eqref{eq2:thm2} imply that
\bqq
 \cE_b(u)\ge \frac{b}{2} \left( \int_{\R^2} |\nabla u|^2 \right)^2-C_p-\frac{1}{2b}\left(\frac{a}{a^*}\right)^2.
\eqq
Therefore, $E(b)>-\infty$. Moreover, if $\{u_n\}$ is a minimizing sequence  for $E(b)$, then it is bounded in $H^1(\R^2)$. 

Recall that $Q_0=Q/\|Q\|_{L^2}$. From \eqref{eq:a*-Q}, we have
\bq \label{eq:Q0}
1=\int_{\R^2} |Q_0|^2 = \int_{\R^2}|\nabla Q_0|^2 = \frac{a^*}{2} \int_{\R^2}|Q_0|^4.
\eq

Let $v_\tau(x)=\tau Q_0(\tau x)$. From \eqref{eq:GN} and \eqref{eq:Q0}, we obtain  
\bqq
E(b)\le \cE_b(v_\tau)= \tau^p\left(b\tau^{4-p} - \tau^{2-p}\left(\frac{a}{a^*}-1\right) - \int_{\R^2}  \frac{|Q_0(x)|^2}{|x|^p} \d x\right). 
\eqq
Since this holds for all $\tau>0$, thus
\bq\label{thm1:tau}
E(b)\le \inf_{\tau>0}\left[ \tau^p\left(b\tau^{4-p} - \tau^{2-p}\left(\frac{a}{a^*}-1\right) - \int_{\R^2}  \frac{|Q_0(x)|^2}{|x|^p} \d x\right)\right]<0. 
\eq

We denote $f^*$ to be the symmetric-decreasing rearrangement of a non-negative measurable function $f$. As discussed, in consideration of \eqref{eq:GP}, we can always restrict to $u\ge 0$. By P\'olya-Szeg\"o  inequality and the well-know rearrangement inequalities, see e.g. \cite[Chapter 3]{LieLos-01}, we have
\bq\label{eq:rea1}
\int_{\R^2} |\nabla u|^2 \ge \int_{\R^2} |\nabla u^*|^2,
\eq
\bq \label{eq:rea2}
\int_{\R^2} \frac{|u|^2}{|x|^p}\le \int_{\R^2} \frac{(|u|^2)^*}{|x|^p}=\int_{\R^2} \frac{(u^*)^2}{|x|^p},
\eq
\bq\label{eq:rea3}
\|u\|_q=\|u^*\|_q,~~ \forall 1\le q < \infty.
\eq
Here in the second bound we used 
\bqq
(|u|^2)^*=(u^*)^2.
\eqq

From \eqref{eq:rea1},\eqref{eq:rea2},\eqref{eq:rea3} we have $\cE_b(u) \ge \cE_b(u^*)$, and hence we can find a minimizing sequence $\{u_n\}$ such that $u_n=u_n^*$ , i.e $u_n$ is radially symmetric decreasing.

Denote by $H^1_{radial}(\R^2)$ the subspace of $H^1(\R^2)$ formed by the radial functions. Since $u_n$ is bounded in $H^1_{radial}(\R^2)$, after passing to a subsequence if necessary, we can assume that $u_n \rightharpoonup  w $ weakly in $H^1_{radial}(\R^2)$. Moreover, it is well-known, see \cite{STRAUSS} or \cite[Theorem A.I']{BerLio-83}, that the embedding $H^1_{radial}(\R^2)\hookrightarrow  L^q(\R^2)$ is compact for all $2< q< +\infty$. Thus 
\bq\label{eq:strongLq}
u_n\to w \text{ strongly in $L^q (\R^2)$ for all $2< q< +\infty$}. 
\eq
Since $u_n$ is non-negative and radially symmetric decreasing, it follows that $w$ is also non-negative and radially symmetric decreasing. 

We will prove that $w$ is a minimizer for $E(b)$. 

Since $\nabla u_n \rightharpoonup \nabla w$ weakly in $L^2(\R^2)$, we have
\bqq
\int_{\R^2} |\nabla u_n|^2 = \int_{\R^2}|\nabla w|^2+\int_{\R^2}|\nabla (w-u_n)|^2+o(1)_{n\to \infty}.
\eqq
Moreover, since $|u_n|^2 \rightharpoonup |w|^2$ weakly in $L^s(\R^2)$ for all $1<s<\infty$ and 
\bqq
V(x)=-|x|^{-p}=-|x|^{-p}\chi_{\{|x|\le 1\}}-|x|^{-p}\chi_{\{|x|\ge 1\}}\in L^{1/p+1/2}+L^{4/p},
\eqq
we have
\bqq
\int_{\R^2} V|u_n|^2=\int_{\R^2} V|w|^2 +o(1)_{n\to \infty}.
\eqq
On the other hand, from \eqref{eq:strongLq}, we get
\bqq
\int_{\R^2} |u_n|^4=\int_{\R^2} |w|^4 +o(1)_{n\to \infty}.
\eqq
We conclude that 
\bq\label{theo1:eb}
E(b)=\liminf\limits_{n\to \infty} \cE_b(u_n)\ge \cE_b(w).
\eq
It remains to show that $\|w\|_{L^2}=1$. We have $\|w\|_{L^2}\le 1$ since $\|u_n\|_{L^2}=1$ and $u_n\rightharpoonup w$ weakly in $L^2(\R^2)$. Moreover, $w\ne 0$ since $E(b)<0$ by \eqref{thm1:tau}. Thus if we let  $\ell=\|w\|_{L^2}^2$ and $u_\ell(x)=w(\ell^{1/2}x)$, then $0<\ell\le 1$ and $\|u_\ell\| _{L^2}=1$. We can estimate
\begin{align}\label{theo1:uell}
E(b)\le \cE_b(u_{\ell})&=\int |\nabla w|^2 +b\left(\int |\nabla w|^2\right)^2 +\frac{1}{\ell^{1-p/2}}\int V(x)|w(x)|^2\d x -\frac{a}{2\ell}\int |w|^4\nn\\
&=\cE_b(w)+\left(\frac{1}{\ell^{1-p/2}}-1\right)\int V(x)|w(x)|^2\d x -\frac{a}{2}\left(\frac{1}{\ell}-1\right)\int |w|^4
\end{align}
It follows from \eqref{theo1:eb} and \eqref{theo1:uell} that 
\begin{align*}
0&\le \left(\frac{1}{\ell^{1-p/2}}-1\right)\int V(x)|w(x)|^2\d x -\frac{a}{2}\left(\frac{1}{\ell}-1\right)\int |w|^4\\
&=-\left(\frac{1}{\ell^{1-p/2}}-1\right)\int \frac{|w(x)|^2}{|x|^p}\d x -\frac{a}{2}\left(\frac{1}{\ell}-1\right)\int |w|^4.
\end{align*}
Combining this with $0<\ell\le 1$ and $0<p<2$, we conclude that $\ell=1$. Therefore, $w$ is a minimizer for $E(b)$. 
 \end{proof}
\section{Blow-up behavior}\label{Blow-up behavior}

In this section we prove Theorem \ref{thm:2} and Theorem \ref{thm:3}. 

First, we prove the sharp upper bound on $E(b)$.

\begin{lemma}[Energy upper bound] \label{lem:Energy upper bound}
For any $b>0$, we have \bqq
bE(b)\le -\frac{1}{4}\left(\frac{a}{a^*}-1\right)^2, ~~ if~~ a>a^*
\eqq
and 
\bqq
b^{p/(4-p)}E(b)\le \inf_{\lambda>0}\left(\lambda^4-\lambda^p\int_{\R^2}\frac{|Q_0(x)|^2}{|x|^p}\d x\right), ~~ if ~~ a=a^*.
\eqq
\end{lemma}

\begin{proof} Let $u_\ell(x)=\ell Q_0(\ell x)$. From \eqref{eq:GN} and \eqref{eq:Q0}, we have 
\bq\label{eq1:Energy upper bound}
E(b)\le \cE_b(u_\ell)= b\ell^4 - \ell^2 \left(\frac{a}{a^*}-1\right) + \ell^p \int_{\R^2} V(x) |Q_0(x)|^2 \d x. 
\eq

$\bullet$ {\bf Case 1:} If $a>a^*$, then we can simply use $\int_{\R^2} V(x) |Q_0(x)|^2 \d x \le 0$ to find that 
$$
E(b) \le b\ell^4 -  \ell^2 \left(\frac{a}{a^*}-1\right). 
$$
Since this holds for all $\ell>0$, by optimizing over $\ell>0$ we obtain 
$$
E(b) \le \inf_{\ell>0} \Big[ b \ell^4 - \ell^2 \left(\frac{a}{a^*}-1\right)\Big] = -\frac{1}{4b} \left(\frac{a}{a^*}-1\right)^2 .
$$
Thus \bqq
bE(b)\le -\frac{1}{4}\left(\frac{a}{a^*}-1\right)^2.
\eqq

$\bullet$ {\bf Case 2:} If $a=a^*$, then from \eqref{eq1:Energy upper bound}, we have
$$
E(b) \le b\ell^4 - \ell^p \int_{\R^2} \frac{|Q_0(x)|^2}{ |x|^{p}} \d x.
$$
Since this holds for all $\ell>0$, we obtain
\begin{align*}
E(b) &\le \inf_{\ell>0} \Big( b \ell^4  - \ell^p \int_{\R^2} \frac{|Q_0(x)|^2}{ |x|^{p}} \d x \Big) \\
&= b^{-p/(4-p)} \inf_{\lambda >0} \Big( \lambda^4  -\lambda^p \int_{\R^2} \frac{|Q_0(x)|^2}{ |x|^{p}} \d x \Big). 
\end{align*}
Thus 
\bqq
b^{p/(4-p)}E(b)\le \inf_{\lambda>0}\left(\lambda^4-\lambda^p\int_{\R^2}\frac{|Q_0(x)|^2}{|x|^p}\d x\right).
\eqq
\end{proof}

In the following step, we provide some rough estimates for the total energy $E(b)$, as well as the kinetic and potential energies of the minimizer $u_b$. These bounds will serve as first inputs and will be improved later in the proof of Theorem \ref{thm:2}.   

\begin{lemma}[Rough energy estimates]\label{lem:ev} Let $u_b$ be a minimizer for $E(b)$. When $a=a^*$ we have
\bqq 
-C^{-1} b ^{-p/(4-p)}\ge E(b)\ge \int_{\R^2} V |u_b|^2 \ge -Cb^{-p/(4-p)},
\eqq
and 
\bqq
\int_{\R^2} |\nabla u_b|^2 \le Cb^{-2/(4-p)}.
\eqq
We always denote by $C\ge 1$ a general constant independent of $b$.
\end{lemma}
\begin{proof}
From Lemma \ref{lem:Energy upper bound} and the Gagliardo-Nirenberg inequality \eqref{eq:GN}, we obtain
\bq\label{eqfirst:lem}
-C^{-1} b ^{-p/(4-p)}\ge E(b)\ge \int_{\R^2} V |u_b|^2.
\eq
From Lemma \ref{lem:4} and \eqref{eq:GN}, we have 
\begin{align*}
E(b)+\int_{\R^2} V|u_b|^2 &= \int_{\R^2} |\nabla u_b|^2+b\left(\int_{\R^2} |\nabla u_b|^2\right)^2+2\int_{\R^2} V|u_b|^2-\frac{a^*}{2}\int_{\R^2} |u_b|^4 \\
&\ge b\left(\int_{\R^2} |\nabla u_b|^2\right)^2-C\eps^{-p/(2-p)}-2\eps \int_{\R^2} |\nabla u_b|^2. 
\end{align*}
Moreover, by AM-GM inequality, we have 
\bqq
\frac{b}{2}\left(\int_{\R^2} |\nabla u_b|^2\right)^2 - 2\eps \int_{\R^2} |\nabla u_b|^2 \ge -\frac{2\eps^2}{b}.
\eqq
Thus, from the above estimates, we obtain
\bqq
E(b)+\int_{\R^2} V|u_b|^2 \ge \frac{b}{2}\left(\int_{\R^2} |\nabla u_b|^2\right)^2 -\frac{2\eps^2}{b}-C\eps^{-p/(2-p)}.
\eqq
Using this with $\eps=b^{(2-p)/(4-p)}$, we obtain
\bq\label{eqlast:lem5}
E(b)+\int_{\R^2} V|u_b|^2 \ge \frac{b}{2}\left(\int_{\R^2} |\nabla u_b|^2\right)^2 -C b^{-p/(4-p)}.
\eq
From \eqref{eqfirst:lem} and \eqref{eqlast:lem5}, we conclude
\bqq
\int_{\R^2} V |u_b|^2 \ge -Cb^{-p/(4-p)} ~~ \text{and} ~~ \int_{\R^2} |\nabla u_b|^2 \le Cb^{-2/(4-p)}.
\eqq
\end{proof}

\begin{proof}[Proof of Theorem \ref{thm:2}] 
We will denote $\eps_b=b^{1/(4-p)}$. Note that $\eps_b\to 0$ when $b\to 0$.

{\bf Step 1: Extracting the limit.} 
From Lemma \ref{lem:ev} we have
\bq\label{eq1:thm2}
\eps_b^2 \int_{\R^2}|\nabla u_b|^2 \le C ~~ \text{and}~~ \eps_b^p \int_{\R^2}V|u_b|^2 \le -C^{-1}.
\eq
Define $w_b(x)=\eps_bu_b(\eps_b x)$. We have $w_b\in H^1(\R^2)$ and $\|w_b\|_{L^2}=1$. Moreover, from \eqref{eq1:thm2}, we have 
\bqq
\int_{\R^2}|\nabla w_b|^2= \eps_b^2 \int_{\R^2}|\nabla u_b|^2 \le C.
\eqq
Thus $w_b$ is bounded in $H^1(\R^2)$. By Sobolev's embedding, after passing to a subsequence if necessary, we can assume that $w_b$ converges to a function $w$ weakly in $H^1(\R^2)$ and pointwise. 

Since $|w_b|^2 \rightharpoonup |w|^2$ weakly in $L^s(\R^2)$ for all $1<s<\infty$ and $|x|^{-p}=|x|^{-p}\chi_{\{|x|\le 1\}}+|x|^{-p}\chi_{\{|x|\ge 1\}}\in L^{1/p+1/2}+L^{4/p}$, \eqref{eq1:thm2} implies that
\bqq
C^{-1}\le \eps_b^p \int_{\R^2}\frac{|u_b(x)|^2}{|x|^p}\d x=\int_{\R^2}\frac{|w_b(x)|^2}{|x|^p}\d x \to \int_{\R^2}\frac{|w(x)|^2}{|x|^p} \d x.
\eqq
We conclude $w \not\equiv 0$.

{\bf Step 2: Relating $w$ and $Q_0$.} Next, we prove that $w$ is an optimizer for the Gagliardo-Nirenberg inequality \eqref{eq:GN}. From \eqref{eq:GN} and Lemma \ref{lem:ev}, we have
\bqq
-C^{-1}\eps_b^{-p}\ge E(b)\ge \int_{\R^2} V |u_b|^2 +b\left(\int_{\R^2} |u_b|^2\right)^2\ge -C\eps_b^{-p}.
\eqq
Since $\eps_b\to 0$ and $0<p<2$, we have
\begin{align}\label{eq1:step2}
0&= \lim_{b\to 0}\eps_b^2\left(E(b)-\int_{\R^2} V |u_b|^2 -b\left(\int_{\R^2} |u_b|^2\right)^2\right)\nn\\
&=\lim_{b\to 0}\eps_b^2\left(\int_{\R^2}|\nabla u_b|^2-\frac{a^*}{2}\int_{\R^2}|u_b|^4\right)\nn\\
&=\lim_{b\to 0}\left(\int_{\R^2}|\nabla w_b|^2-\frac{a^*}{2}\int_{\R^2}|w_b|^4\right).
\end{align}
Since $w_b$ converges to $w$ weakly in $H^1(\R^2)$, we have 
\bq\label{eq2:step2}
\int_{\R^2}|\nabla w_b|^2=\int_{\R^2}|\nabla w|^2+\int_{\R^2}|\nabla (w_b-w)|^2 +o(1)_{b\to 0}.
\eq
From Brezis-Lieb's lemma \cite{BreLie-83}, we have
\bq\label{eq3:step2}
\int_{\R^2}| w_b|^4=\int_{\R^2}| w|^4+\int_{\R^2}|w_b-w|^4 +o(1)_{b\to 0}.
\eq
Combining \eqref{eq1:step2},\eqref{eq2:step2},\eqref{eq3:step2}, we obtain
\bq\label{eq4:step2}
0=\lim_{b\to 0} \left(\int_{\R^2}|\nabla w|^2+\int_{\R^2}|\nabla (w_b-w)|^2 - \frac{a^*}{2}\int_{\R^2}| w|^4-\frac{a^*}{2}\int_{\R^2}|w_b-w|^4\right).
\eq
On the other hand, by the Gagliardo-Nirenberg inequality \eqref{eq:GN}, we have 
\bq\label{eq5:step2}
\int_{\R^2}|\nabla w|^2- \frac{a^*}{2}\int_{\R^2}| w|^4\ge (1-\|w\|_{L^2}^2)\int_{\R^2} |\nabla w|^2
\eq
and 
\bq\label{eq6:step2}
\int_{\R^2}|\nabla (w_b-w)|^2- \frac{a^*}{2}\int_{\R^2}| w_b-w|^4\ge (1-\|w_b-w\|_{L^2}^2)\int_{\R^2} |\nabla (w_b-w)|^2.
\eq
Moreover, from Brezis-Lieb's lemma \cite{BreLie-83}, we have
\begin{align}\label{eq7:step2}
1=\int_{\R^2}| w_b|^2&=\int_{\R^2}| w|^2+\int_{\R^2}|w_b-w|^2 +o(1)_{b\to 0}\nn\\
&=\|w\|_{L^2}^2+\|w_b-w\|_{L^2}^2+o(1)_{b\to 0}.
\end{align}
Combining \eqref{eq6:step2}, \eqref {eq7:step2} and the fact that $\|\nabla(w_b-w)\|_{L^2}$ is bounded (since $w_b$ is bounded in $H^1(\R^2)$), we have 
\bq\label{eq8:step2}
\int_{\R^2}|\nabla (w_b-w)|^2- \frac{a^*}{2}\int_{\R^2}| w_b-w|^4 \ge \|w\|_{L^2}^2\int_{\R^2} |\nabla (w_b-w)|^2+o(1)_{b\to 0}.
\eq
Therefore, from \eqref{eq4:step2},\eqref{eq5:step2},\eqref{eq8:step2}, we have 
\bq\label{eq9:step2}
\limsup_{b\to 0} \left((1-\|w\|_{L^2}^2)\int_{\R^2} |\nabla w|^2+\|w\|_{L^2}^2\int_{\R^2} |\nabla (w_b-w)|^2\right)\le 0.
\eq
Since $\|w_b\|_{L^2}=1$ and $w_b \to w$ weakly in $L^2(\R^2)$, we have 
$\|w\|_{L^2} \le 1$. Moreover, $w \not\equiv 0$ by the conclusion in Step 1. Thus \eqref{eq9:step2} implies that $\|w\|_{L^2}=1$ and $\|\nabla(w_b-w)\|_{L^2}\to 0$. Therefore, $w_b\to w$ strongly in $H^1(\R^2)$. Combining this and \eqref{eq1:step2}, we obtain that $w$ is an optimizer for the Gagliardo-Nirenberg inequality \eqref{eq:GN}.

Since $Q_0$ is the unique optimizer for \eqref{eq:GN} up to translations and dilations, we conclude that 
\bqq
w=\beta Q_0(\beta x-x_0),
\eqq
for some $\beta>0$ and $x_0\in \R^2$. The values of $\beta$ and $x_0$ will be determined below.

{\bf Step 3: Energy low bound.} We have
\bq\label{eq2:step3}
\int_{\R^2} V|u_b|^2=-\int_{\R^2}\frac{|u_b(x)|^2}{|x|^p} \d x = -\eps_b^{-p}\int_{\R^2} \frac{|w_b(x)|^2}{|x|^p} \d x.
\eq
By H\"older's and Sobolev's inequalies, we have
\begin{align} \label{eq1:step3}
\left|\int_{\R^2} \frac{|w_b|^2-|w|^2}{|x|^p}\right|&\le \left(\int_{\R^2} \frac{|w_b-w|^2}{|x|^p}\right)^{1/2}\left(\int_{\R^2} \frac{|w_b+w|^2}{|x|^p}\right) ^{1/2} \nn\\
&\le C\|w_b-w\|_{H^1}\|w_b+w\|_{H^1}.
\end{align}
Since $w_b\to w$ strongly in $H^1(\R^2)$ and $\|w_b+w\|_{H^1}$ is bounded, we have 
\bqq
C\|w_b-w\|_{H^1}\|w_b+w\|_{H^1}\to 0.
\eqq
Thus \eqref{eq1:step3} implies that 
\bq \label{eq3:step3}
\int_{\R^2} \frac{|w_b(x)|^2}{|x|^p}\d x=\int_{\R^2} \frac{|w(x)|^2}{|x|^p}\d x+o(1)_{b\to 0}.
\eq
Moreover, we have
\bq\label{eq4:step3}
\int_{\R^2}|\nabla u_b|^2=\eps_b^{-2}\int_{\R^2}|\nabla w_b|^2=\eps_b^{-2}\left(\int_{\R^2}|\nabla w|^2+o(1)_{b\to 0}\right).
\eq
From \eqref{eq2:step3},\eqref{eq3:step3},\eqref{eq4:step3} and the Gagliardo-Nirenberg inequality \eqref{eq:GN}, we obtain
\begin{align}\label{eq5:step3} 
E(b)&=\int_{\R^2} |\nabla u_b|^2  + b \left( \int_{\R^2} |\nabla u_b|^2  \right)^2 + \int_{\R^2} V |u_b|^2  - \frac{a^{*}}{2} \int_{\R^2} |u_b|^4\nn\\
&\ge b\eps_b^{-4} \left(\int_{\R^2}|\nabla w|^2+o(1)_{b\to 0}\right)^2 -\eps_b^{-p}\left(\int_{\R^2} \frac{|w(x)|^2}{|x|^p}\d x+o(1)_{b\to 0}\right).
\end{align}
Recall that $\eps_b=b^{1/(4-p)}$. Thus $b\eps_b^{-4}=\eps_b^{-p}=b^{-p/(4-p)}$. Combining this with \eqref{eq5:step3}, we conclude 
\bq\label{eq6:step3} 
\liminf_{b\to 0}b^{p/(4-p)}E(b)\ge \left(\int_{\R^2}|\nabla w|^2\d x\right)^2-\int_{\R^2} \frac{|w(x)|^2}{|x|^p}\d x.
\eq

{\bf Step 4: Conclusion.} Now we use $w(x)=\beta Q_0(\beta x-x_0)$. From \eqref{eq:Q0} and \eqref{eq6:step3}, we have
\bqq
\liminf_{b\to 0}b^{p/(4-p)}E(b)\ge \beta^4-\beta^p\int_{\R^2} \frac{|Q_0(x-x_0)|^2}{|x|^p}\d x.
\eqq
Moreover, recall that $Q_0$ is radially symmetric decreasing and $\frac{1}{|x|^p}$ is radially symmetric strictly decreasing. Thus, by well-know rearrangement inequalities, see e.g. \cite[Chapter 3.4]{LieLos-01}, we deduce that
\bqq
\int_{\R^2}\frac{|Q_0(x-x_0)|^2}{|x|^p}\d x\le \int_{\R^2}\frac{|Q_0(x)|^2}{|x|^p},
\eqq
with equality if and only if $x_0=0$. Thus
\bqq
\liminf_{b\to 0}b^{p/(4-p)}{E(b)}\ge \beta^4-\beta^p\int_{\R^2} \frac{|Q_0(x)|^2}{|x|^p}\d x.
\eqq
On the other hand, we have proved in Lemma \ref{lem:Energy upper bound} that 
\bqq
b^{p/(4-p)}E(b)\le \inf_{\lambda>0}\left(\lambda^4-\lambda^p\int_{\R^2}\frac{|Q_0(x)|^2}{|x|^p}\d x\right),~~\forall b>0.
\eqq
Therefore, we conclude that
\bqq
\lim_{b\to 0}b^{p/(4-p)}E(b)= \inf_{\lambda>0}\left(\lambda^4-\lambda^p\int_{\R^2}\frac{|Q_0(x)|^2}{|x|^p}\d x\right).
\eqq
Moreover $x_0=0$ and $\beta$ is the optimal value in
\begin{align*}
&\inf_{\lambda>0}\left(\lambda^4-\lambda^p\int_{\R^2}\frac{|Q_0(x)|^2}{|x|^p}\d x\right)\\
&=\left(\int_{\R^2}\frac{|Q_0(x)|^2}{|x|^p}\d x\right)^{4/(4-p)}\left[\left(\frac{p}{4}\right)^{4/(4-p)}-\left(\frac{p}{4}\right)^{p/(4-p)}\right].
\end{align*}
This means 
\bqq
\beta= \left(\frac{p}{4} \int_{\R^2} \frac{|Q_0(x)|^2}{|x|^p} \d x\right)^{1/(4-p)}.
\eqq
The proof of Theorem \ref{thm:2} is finished.
\end{proof}

Finally, we provide:
\begin{proof}[Proof of Theorem \ref{thm:3}]  We have proved in Lemma \ref{lem:Energy upper bound} that 
\bqq
bE(b)\le -\frac{1}{4}\left(\frac{a}{a^*}-1\right)^2,~~ \forall b>0,a>a^*.
\eqq
Therefore, it remains to show that
\bqq
\liminf \limits_{b\to 0} bE(b) \ge -\frac{1}{4}\left(\frac{a}{a^*}-1\right)^2.
\eqq
Let $u_b$ be the minimizer for $E(b)$. By Lemma \ref{lem:4}
, we have 
\bqq
\eps\int_{\R^2} |\nabla u_b|^2+\int_{\R^2} V |u_b|^2 \ge -C_\eps.
\eqq
Thus 
\bqq
E(b) \ge  b \left( \int_{\R^2} |\nabla u_b|^2 \right)^2 + (1-\eps) \int_{\R^2} |\nabla u_b|^2 - \frac{a}{2}\int_{\R^2} |u_b|^4 -C_\eps .
\eqq
Combining this with \eqref{eq:GN}, we obtain
\bqq
E(b)\ge b \left( \int_{\R^2} |\nabla u_b|^2 \right)^2 - \left(\frac{a}{a^*} +\eps -1\right) \int_{\R^2} |\nabla u_b|^2 - C_\eps. 
\eqq
By completing the square, we find that 
$$
E(b) \ge  -\frac{1}{4b} \left(\frac{a}{a^*} +\eps -1\right)^2 - C_\eps.  
$$
Thus 
\bqq
\liminf_{b\to 0}bE(b) \ge -\frac{1}{4} \left(\frac{a}{a^*} +\eps -1\right)^2.
\eqq
Since this holds for all $\eps>0$, we obtain that 
\bqq
\liminf \limits_{b\to 0} bE(b) \ge -\frac{1}{4}\left(\frac{a}{a^*}-1\right)^2.
\eqq
The proof of Theorem \ref{thm:3} is complete.
\end{proof}

\end{document}